\documentclass[10pt]{amsart}
\usepackage{amsfonts,amssymb,mathtools,amsmath, array, amsthm}
\usepackage{geometry}
\usepackage{hyperref}
\usepackage{tikz}
\usepackage{textcomp, color}
\usepackage[capitalise]{cleveref}

\theoremstyle{plain}

\newtheorem{theo}{Theorem}[section]

\newtheorem{theorem}[theo]{Theorem}
\newtheorem{lemma}[theo]{Lemma}

\newtheorem{proposition}[theo]{Proposition}
\newtheorem{corollary}[theo]{Corollary}

\newtheorem{remark}[theo]{Remark}
\newtheorem{definition}[theo]{Definition}

\DeclareMathOperator{\Aut}{Aut}

\DeclareMathOperator{\Br}{Br}
\DeclareMathOperator{\SmallSkewbrace}{SmallSkewbrace}

\newtheorem{pro}[theo]{Proposition}
\newtheorem{exa}[theo]{Example}

\begin{document}
\baselineskip13.2pt
\title {THE BAER TRANSFORM FOR SKEW BRACES}

\author{G\"ulİn Ercan}
\address{Department of Mathematics, Middle East Technical University, Ankara, Turkey}
\email{ercan@metu.edu.tr}

\author{\c{S}\"ukran G\"ul}
\address{Department of Mathematics, Middle East Technical University, Ankara, Turkey}
\email{gsukran@metu.edu.tr}

\author{\.{I}smaİl \c{S}.\ G\"ulo\u{g}lu}
\address{Department of Mathematics, Do\u{g}u\c{s} University, \.{I}stanbul, Turkey}
\email{iguloglu@dogus.edu.tr}

\author{M.\ YASİR KIZMAZ}
\address{Department of Mathematics, Bilkent University, Ankara, Turkey}
\email{yasirkizmaz@bilkent.edu.tr}

\begin{abstract}
We introduce and study the Baer transform for finite skew braces whose
additive group has odd order and nilpotency class at most $2$. Starting from
Baer's classical construction for nilpotent groups of odd order and nilpotency class $2$, we replace the
additive group $(X,+)$ by an abelian group $(X,\oplus)$ on the same underlying
set and prove that
\[
\Br(X)=(X,\oplus,\cdot)
\]
is a finite skew brace of abelian type. We show that the Baer transform preserves
automorphisms, strong left ideals, ideals and central ideals, and we establish
its compatibility with quotients. We then compare structural properties of
$X$ and $\Br(X)$. As applications, we show that direct product
decompositions into ideals are preserved by the Baer transform, giving a
criterion for indecomposability of skew braces. Finally, for finite
$p$-skew braces with $p$ odd, we apply Thompson critical subgroups together
with the Baer transform to embed suitable automorphism groups into
automorphism groups of skew braces of abelian type.
\end{abstract}

	

\maketitle 
    
\section*{Introduction}

Skew braces were introduced by Guarnieri and Vendramin \cite{GV} as a generalization
of Rump's braces and have become an important algebraic tool in the study of
set-theoretic solutions of the Yang--Baxter equation. A skew (left) brace is a
set $X$ endowed with two group operations $+$ and $\cdot$ satisfying
\[
a\cdot(b+c)=a\cdot b-a+a\cdot c
\]
for all $a,b,c\in X$. If the additive group $(X,+)$ is abelian, one obtains
the classical notion of a left brace, or equivalently, a skew brace of
abelian type.

Many structural questions for skew braces are considerably more difficult
than in the abelian-type case, because the additive group need not be
abelian. The purpose of this paper is to show that, when the additive group
has odd order and nilpotency class at most $2$, one can associate to $X$ a
canonical brace of abelian type with the same multiplication. This is achieved
by applying Baer's classical trick to the additive group.

Let $(G,+)$ be a group of odd order and nilpotency class at most $2$. Baer's
construction defines a new operation
\[
x\oplus y=x+y+\frac{1}{2}[y,x]_+,
\]
where
\[
[y,x]_+=y+x-y-x,
\]
and where $\frac{1}{2}z$ denotes the unique element $y$ satisfying $y+y=z$. The resulting
group $(G,\oplus)$ is abelian. Moreover, this construction preserves the
identity element, inverses, element orders, and automorphisms of $(G,+)$.

We apply this construction to the additive group of a skew brace
$X=(X,+,\cdot)$. Our first main result proves that
\[
\Br(X)=(X,\oplus,\cdot)
\]
is a brace of abelian type. We call $\Br(X)$ the Baer transform of $X$.
Thus the Baer transform replaces the additive group of $X$ by an abelian
group, while keeping the multiplicative group unchanged.

The usefulness of the construction comes from the fact that it retains a
large amount of the original brace structure. We prove that every
automorphism of $X$ is an automorphism of $\Br(X)$, and that every strong
left ideal, ideal and central ideal of $X$ remains respectively a strong left
ideal, ideal and central ideal of $\Br(X)$. We also prove that the Baer
transform is compatible with quotients, that is
\[
\Br(X/I)=\Br(X)/I
\]
for every ideal $I$ of $X$.

These basic properties allow us to compare several nilpotency and solvability-type conditions for $X$ and $\Br(X)$. The fact that $X$ is left nilpotent if and only if $\Br(X)$ is left nilpotent follows directly from \cite[Theorem 4.8]{CSV}. In this paper, we find an upper bound for the left nilpotency class of one in terms of that of the other. We also prove that right nilpotency, central nilpotency, solvability and supersolvability passes from $X$ to $\Br(X)$, the converses
do not hold in general. For the bi-skew property, the
Baer transform is more delicate: if $X$ is bi-skew, then $\Br(X)$ is bi-skew
if and only if, for every $a\in X$, the map
\[
g_a:X\longrightarrow X,\qquad
g_a(b)=b+\frac{1}{2}[b,a]_+
\]
is an automorphism of the multiplicative group $(X,\cdot)$. This gives an
explicit obstruction to preserving the bi-skew property.

The last part of the paper gives two applications. First, we prove that
internal direct product decompositions of $X$ into ideals are preserved by
the Baer transform. Consequently, if $\Br(X)$ is indecomposable,
then $X$ is indecomposable. This gives a useful criterion for
indecomposability of skew braces by passing to an abelian type brace.

Secondly, we apply the Baer transform to automorphisms of finite $p$-skew
braces. Let $p$ be odd and suppose that $(X,+)$ is a finite $p$-group. If
$C$ is a Thompson critical subgroup of $(X,+)$, then $C$ has nilpotency class
at most $2$, and hence its Baer transform $\Br(C)$ is defined. We use this
to show that, under a natural condition on the $p$-core of $\Aut(X)$, the
restriction map embeds $\Aut(X)$ into $\Aut(\Br(C))$. In particular, if
$\Aut(\Br(C))$ is a $p$-group, then $\Aut(X)$ has no non-trivial
$p'$-subgroups.

The paper is organized as follows. In Section~1 we recall Baer's trick and
construct the Baer transform of a skew brace. In Section~2 we prove the basic
properties of the transform. In Section~3 we
study brace-theoretic properties preserved by the Baer transform, including
left and right nilpotency, two-sidedness, bi-skewness, central nilpotency,
solvability and supersolvability. In Section~4 we give applications to direct
decompositions and to automorphisms of finite $p$-skew braces via Thompson
critical subgroups.

\textbf{Throughout the paper all skew braces are finite. We use standard notation
and terminology for skew braces as in \cite{CV}.
}

\section{Baer's trick and the Baer transform}

We begin by recalling Baer's classical construction for nilpotent groups of
class $2$ and odd order{\cite[Lemma 4.37]{isa}}.

\begin{lemma}[Baer's trick]\label{lem:baer}
Let $(G,+)$ be a nilpotent group of odd order with nilpotency class at most
$2$. Then there exists an operation (called Baer addition) $x\oplus y$ defined for all
$x,y\in G$ such that $(G,\oplus)$ is an abelian group. Moreover, the following
hold:
\begin{enumerate}
\item the order of $x$ in $(G,+)$ is equal to the order of $x$ in
$(G,\oplus)$ for all $x\in G$;
\item the inverse of $x$ in $(G,+)$ coincides with the inverse of $x$ in
$(G,\oplus)$ for all $x\in G$;
\item the identity element of $(G,+)$ coincides with that of $(G,\oplus)$;
\item every automorphism of $(G,+)$ is also an automorphism of $(G,\oplus)$.
\end{enumerate}
\end{lemma}

\begin{remark}\label{rem:baer-formula}
In the situation of Baer's trick, we use the following notation. For $x\in G$,
the symbol $\dfrac{1}{2}  x$ denotes the unique element $y$ satisfying $y+y=x$. If
\[
[y,x]_+=y+x-y-x
\]
denotes the commutator in the group $(G,+)$, then
\[
x\oplus y=x+y+\frac{1}{2}[y,x]_+.
\]
Moreover, if $x$ and $y$ commute in $(G,+)$, then
\[
x\oplus y=x+y.
\]
If $u$ and $v$ commute in $(G,+)$, then
\[
\frac{1}{2}(u+v)=\frac{1}{2}u+\frac{1}{2}v.
\]
\end{remark}

Now we apply Baer's trick to the additive group of a skew brace.

\begin{theorem}\label{thm:baer-transform}
Let $X=(X,+,\cdot)$ be a skew brace such that the additive group
$(X,+)$ has odd order and nilpotency class at most $2$. Let $\oplus$ be the
Baer addition associated to the group $(X,+)$. Then
\[
\operatorname{Br}(X)=(X,\oplus,\cdot)
\]
is a brace, that is, a skew brace of abelian type.
\end{theorem}

\begin{proof}
We need to prove the compatibility condition for the operations $\oplus$ and
$\cdot$. Equivalently, for each $a\in X$, we need to prove that the map
\[
\mu_a:(X,\oplus)\longrightarrow (X,\oplus),
\qquad
\mu_a(b)=\ominus a\oplus (a\cdot b),
\]
is an automorphism of $(X,\oplus)$.

By Lemma~\ref{lem:baer}, the inverse of $a$ with respect to $\oplus$ is the
inverse of $a$ in the group $(X,+)$, namely $-a$. Hence
\[
\mu_a(b)=-a\oplus (a\cdot b).
\]
Using the definition of $\oplus$, we obtain
\[
\mu_a(b)=-a+a\cdot b+\frac{1}{2}[a\cdot b,-a]_+.
\]
Using the class two commutator identities in $(X,+)$, we have
\[
[a\cdot b,-a]_+=[a,a\cdot b]_+.
\]
Therefore
\[
\mu_a(b)=\lambda_a(b)+\frac{1}{2}[a,a\cdot b]_+,
\]
where $\lambda_a(b)=-a+a\cdot b$ is the usual lambda map of $X$.

Since $(X,+)$ has
nilpotency class at most $2$, the subgroup $\gamma_2(X,+)$ is central. Hence,
whenever $u,v\in \gamma_2(X,+)$, the elements $u$ and $v$ commute and so $\frac{1}{2}(u+v)=\frac{1}{2}u+\frac{1}{2}v.$ Let $b,c\in X$. Since $\lambda_a$ is an automorphism of $(X,+)$ and all
additive commutators are central, we have
\[
\begin{aligned}
\mu_a(b+c)
&=\lambda_a(b+c)+\frac{1}{2}[a,a\cdot (b+c)]_+ \\
&=\lambda_a(b+c)+\frac{1}{2}[a,a\cdot b-a+a\cdot c]_+ \\
&=\lambda_a(b+c)+\frac{1}{2}([a,a\cdot b]_+ + [a,a\cdot c]_+) \\
&=\lambda_a(b)+\lambda_a(c)
   +\frac{1}{2}[a,a\cdot b]_++\frac{1}{2}[a,a\cdot c]_+ \\
&=\mu_a(b)+\mu_a(c).
\end{aligned}
\]
Thus $\mu_a$ is an endomorphism of $(X,+)$.

We next prove that $\mu_a$ is injective. Let $b\in \ker \mu_a$. Then
\[
\lambda_a(b)+\frac{1}{2}[a,a\cdot b]_+=0.
\]
Since $\frac{1}{2}[a,a\cdot b]_+\in \gamma_2(X,+)\leq Z(X,+)$, it follows that
$\lambda_a(b)\in Z(X,+)$. As
$a\cdot b=a+\lambda_a(b)$,
we have
$ [a,a\cdot b]_+=[a,a+\lambda_a(b)]_+=0$.
Thus, we have $\lambda_a(b)+\frac{1}{2} \ 0=\lambda_a(b)=0$.

Since $\lambda_a$ is an automorphism of $(X,+)$, we obtain $b=0$. Hence
$\mu_a$ is injective. Since $X$ is finite, $\mu_a$ is an automorphism of
$(X,+)$.

By Lemma~\ref{lem:baer}, every automorphism of $(X,+)$ is also an
automorphism of $(X,\oplus)$. Hence
\[
\mu_a\in \operatorname{Aut}(X,\oplus)
\]
for every $a\in X$. Therefore $(X,\oplus,\cdot)$ is a  brace.
\end{proof}

\begin{definition}
Under the hypotheses of Theorem~\ref{thm:baer-transform}, we call $\operatorname{Br}(X)=(X,\oplus,\cdot)$ the Baer transform of $X$.
\end{definition}

\begin{corollary}\label{cor:trivial-baer-brace}
Let $(G,\cdot)$ be a group of odd order with nilpotency class at most $2$.
Then the Baer addition associated to $(G,\cdot)$ turns $G$ into a brace $(G,\oplus,\cdot).$
\end{corollary}

\begin{proof}
Consider the trivial skew brace $(G,\cdot,\cdot)$ and apply
Theorem~\ref{thm:baer-transform}.
\end{proof}

\begin{theorem}\label{thm:baer-trivial-nilpotent}
Let $(G,\cdot)$ be a nilpotent group of odd order with nilpotency class at most
$2$. Let $(G,\oplus,\cdot)$ be the brace obtained from $(G,\cdot)$ by Baer's
trick. Then $(G,\oplus,\cdot)$ is both left nilpotent and right nilpotent of
class at most $2$. More precisely,
\[
G*_{\oplus}G=[G,G]_\cdot
\]
and
\[
G*_{\oplus}(G*_{\oplus}G)=0,
\qquad
(G*_{\oplus}G)*_{\oplus}G=0.
\]
\end{theorem}

\begin{proof}
For $a,b\in G$, the Baer addition is
\[
a\oplus b=a\cdot b\cdot \sqrt{[b,a]_\cdot},
\]
The inverse of $a$ with respect to $\oplus$ is $a^{-1}$. Thus
\[
\mu_a(b)=\ominus a\oplus (a\cdot b)=a^{-1}\oplus (a\cdot b).
\]
Hence
\[
\mu_a(b)=a^{-1}\cdot a\cdot b\cdot \sqrt{[a\cdot b,a^{-1}]_\cdot}
       =b\cdot \sqrt{[a,b]_\cdot},
\]
where the last equality follows from the fact that $[G,G]_\cdot\leq Z(G,\cdot)$.

Since $z=\sqrt{[a,b]_{\cdot}}\in Z(G,\cdot)$, $b$ and $z$ commute in $(G,\cdot)$, and so by Remark \ref{rem:baer-formula},
\[
b\cdot z=b\oplus z.
\]
Therefore, we have
$\mu_a(b)=b\oplus z$.
It then follows that
\[
a*_{\oplus}b=\mu_a(b)\ominus b=(b\oplus z)\ominus b=z=\sqrt{[a,b]_\cdot}.
\]
Consequently,
\[
G*_{\oplus}G
=
\langle \sqrt{[a,b]_\cdot}\mid a,b\in G\rangle_\oplus.
\]
Since $|G|$ is odd, the square-root map is a bijection on $[G,G]_\cdot$.
Hence
\[
G*_{\oplus}G=[G,G]_\cdot.
\]

Now let $c\in [G,G]_\cdot$. Since $[G,G]_\cdot\leq Z(G,\cdot)$, for every
$a\in G$ we have
\[
[a,c]_\cdot=1
\qquad\text{and}\qquad
[c,a]_\cdot=1.
\]
Thus
\[
a*_{\oplus}c=\sqrt{[a,c]_\cdot}=1 \quad \text{and} \quad c*_{\oplus}a=\sqrt{[c,a]_\cdot}=1
\]
Since $1$ is the zero element of $(G,\oplus)$, this gives
\[
G*_{\oplus}(G*_{\oplus}G)=0 \quad  \text{and} \quad   (G*_{\oplus}G)*_{\oplus}G=0
\]
Therefore the brace is both left nilpotent and right nilpotent of class at
most $2$.
\end{proof}


\section{Basic properties of the Baer transform}

We first study the structural properties of the Baer transform. The main
point is that, although the additive group is changed, automorphisms and ideals
of the original skew brace are preserved.

\begin{lemma}\label{lem:automorphisms}
Let $X=(X,+,\cdot)$ and $\operatorname{Br}(X)=(X,\oplus,\cdot)$ be as in
Theorem~\ref{thm:baer-transform}. Then
\[
\operatorname{Aut}(X)\leq \operatorname{Aut}(\operatorname{Br}(X)).
\]
\end{lemma}

\begin{proof}
Let $\alpha\in \operatorname{Aut}(X)$. Then $\alpha$ is an automorphism of the
additive group $(X,+)$ and of the multiplicative group $(X,\cdot)$. By
Lemma~\ref{lem:baer}, every automorphism of $(X,+)$ is also an automorphism of
$(X,\oplus)$. Hence $\alpha$ preserves both $\oplus$ and $\cdot$. Therefore $\alpha\in \operatorname{Aut}(\operatorname{Br}(X)).$
\end{proof}

A key advantage of the Baer transform is that it preserves ideals.

\begin{proposition}\label{prop:ideals-preserved}
Let $X=(X,+,\cdot)$ and $\operatorname{Br}(X)=(X,\oplus,\cdot)$ be as in
Theorem~\ref{thm:baer-transform}. Then the following hold.
\begin{enumerate}
\item Every strong left ideal of $X$ is a strong left ideal of
$\operatorname{Br}(X)$.
\item Every ideal of $X$ is an ideal of $\operatorname{Br}(X)$. Moreover, if $I$ is a central ideal of $X$,
then $I$ is a central ideal of $\Br(X)$.
\end{enumerate}
\end{proposition}

\begin{proof}
Let $I$ be a strong left ideal of $X$. Thus $(I,+)\unlhd (X,+)$,
$(I,\cdot)\leq (X,\cdot)$, and $\lambda_x(I)\subseteq I$ for every $x\in X$.

Let $a,b\in I$. We have $[b,a]_+\in I.$
As $I$ has odd order, $\frac{1}{2}z\in I$ for every $z\in I$.
Hence
$\frac{1}{2}[b,a]_+\in I.
$
Therefore
\[
a\oplus b=a+b+\frac{1}{2}[b,a]_+\in I.
\]
The inverse of $a$ with respect to $\oplus$ is $-a$, and this belongs to $I$.
Thus $(I,\oplus)\leq (X,\oplus)$. Since $(X,\oplus)$ is abelian, we also have
\[
(I,\oplus)\unlhd (X,\oplus).
\]

It remains to prove that $I$ is invariant under the lambda maps of
$\operatorname{Br}(X)$. Let $\mu_x$ denote the lambda map of
$\operatorname{Br}(X)$, so that
\[
\mu_x(y)=\ominus x\oplus (x\cdot y).
\]
From the proof of Theorem~\ref{thm:baer-transform}, for every $i\in I$, we have
\[
\mu_x(i)=\lambda_x(i)+\frac{1}{2}[x,x\cdot i]_+.
\]
Since
$x\cdot i=x+\lambda_x(i)$, we have
$[x,x\cdot i]_+=[x,\lambda_x(i)]_+.$
As $\lambda_x(i)\in I$ and $(I,+)\unlhd (X,+)$, it then follows that
\[
[x,x\cdot i]_+=[x,\lambda_x(i)]_+\in I,
\]
and hence also
\[
\frac{1}{2}[x,x\cdot i]_+\in I.
\]
Since $\lambda_x(i)\in I$, we get $\mu_x(i)\in I$. Therefore $I$ is a strong
left ideal of $\operatorname{Br}(X)$. This proves $(1)$.

Now let $I$ be an ideal of $X$. Then $I$ is a strong left
ideal of $X$ and $(I,\cdot)\unlhd (X,\cdot)$. By $(1)$, $I$ is a strong left
ideal of $\operatorname{Br}(X)$. Since $X$ and $\operatorname{Br}(X)$ have the
same multiplicative group, we still have
$
(I,\cdot)\unlhd (X,\cdot).
$ Therefore $I$ is an ideal of $\operatorname{Br}(X)$.

Finally suppose that $I$ is an ideal contained in the center of $X.$ Let $i\in I$ and $x\in X$. Then we have
\[
i\in Z(X,+),\qquad i\in Z(X,\cdot),\qquad \lambda_i=\operatorname{id}_X.
\]
First, since $i$ and $x$ commute in $(X,+)$,  by Remark \ref{rem:baer-formula}  we have
$i\oplus x=i+x=x+i=x\oplus i$. Thus $i$ is central in $(X,\oplus)$. Next, $i$ is central in $(X,\cdot)$ by assumption, and the multiplication of
$\Br(X)$ is the same as the multiplication of $X$. Therefore
$
i\cdot x=x\cdot i.
$

Finally, we check that $i$ lies in the kernel of the lambda map of $\Br(X)$.
Let $\mu$ denote the lambda map of $\Br(X)$. From the formula for $\mu$ we have
\[
\mu_i(x)=\lambda_i(x)+\frac{1}{2}[i,i\cdot x]_+.
\]
Since $\lambda_i=\operatorname{id}_X$, this becomes
\[
\mu_i(x)=x+\frac{1}{2}[i,i\cdot x]_+.
\]
As $i\in Z(X,+)$, we have
$
[i,i\cdot x]_+=0,
$
and this implies that
$
\mu_i(x)=x,
$
and hence $\mu_i=\operatorname{id}_X$.

Therefore every element of $I$ lies in $Z(\Br(X))$, completing the proof.
\end{proof}

\section{Properties which are invariant under the Baer Transform}

In this section, we explore the structural relationships between the skew braces $X$ and $\Br(X)$.

Let $X$ and $\Br(X)$ be as in Theorem \ref{thm:baer-transform}. Then, recall that for all $a, b\in X$ we have
\[
\mu_a(b)=\lambda_a(b)+\frac{1}{2}[a, \lambda_a(b)]_+
\]
and
\[
a*_{\oplus}b=\mu_a(b) \oplus (-b)= \lambda_a(b)-b+ \frac{1}{2}[a, \lambda_a(b)]_++\frac{1}{2}[-b, \lambda_a(b)]_+
\]

\vspace{5pt}
Since $X$ and $\Br(X)$ have the same multiplicative group and their additive groups are nilpotent, by \cite[Theorem 4.8]{CSV}, $X$ and $\Br(X)$ are left nilpotent if and only if the common multiplicative group is nilpotent. The next theorem gives an upper bound for the left nilpotency class of one in terms of that of the other.

\begin{theo}
\label{theo: left nilpotency}
Let $X$ be a skew brace of odd order such that the additive group $(X, +)$ is a nilpotent group of class at most $2$. Then we have the following:
\begin{enumerate}
    \item 
    If the left nilptoency class of $X$ is $r$, then the left nilpotency class of $\Br(X)$ is at most $2r$.
    \item 
    If the left nilptoency class of $\Br(X)$ is $n$, then the left nilpotency class of $X$ is at most $2n$.
\end{enumerate}
\end{theo}

\begin{proof}
Put $I=\gamma_2(X,+)$. First of all, by induction on $n$, we will prove that 
\begin{equation*}
\Br(X)^{n}\equiv X^{n} \pmod{I}
\end{equation*}
for all $n\geq 2$.
We have
\[
X*X=\langle  \lambda_a(b)-b \mid a, b \in X \rangle_{+} \]
and 
\begin{equation}
	\label{eqn:star product modulo gamma}
	\begin{split}
\Br(X)*_{\oplus}\Br(X) &= \langle \mu_a(b) \oplus (-b) \mid a ,b \in X \rangle_{\oplus}\\
&= \langle  \lambda_a(b)-b+ \frac{1}{2}[a, \lambda_a(b)]_+ +\frac{1}{2}[-b, \lambda_a(b)]_+  \mid a ,b \in X\rangle_{\oplus}\\
&\equiv \langle  \lambda_a(b)-b   \mid a ,b \in X \rangle_+ \pmod{I}\\
& \equiv X*X \pmod{I}
	\end{split}
\end{equation}

Now suppose that for some $n\geq 3$,
\begin{equation*}
\Br(X)^{n-1}\equiv X^{n-1} \pmod{I}
\end{equation*}
 Let $a\in \Br(X)$ and $b \in \Br(X)^{n-1}$. Then by the assumption, we have 
\[
b=c+z
\]
for some $c\in X^{n-1}$ and $z\in I$. Then taking into account that $I$ is lambda-invariant and central, we get 
\[
\begin{aligned}
a*_{\oplus} b&= a*_{\oplus}(c+z)\\
&= \lambda_a(c+z)-(c+z)+ \frac{1}{2}[a, \lambda_a(c+z)]_++\frac{1}{2}[-(c+z), \lambda_a(c+z)]_+\\
&\equiv \lambda_a(c)-c \pmod{I}\\
&\equiv a*c \pmod{I}
\end{aligned}
\]
Thus, for all $n\geq 2$, we have
\begin{equation}
	\label{eqn:higher product congruence}
\Br(X)^n\equiv X^n \pmod{I}
\end{equation}

We will next observe that $*$ operation on $I$ in $X$ coincides with that in $\Br(X)$.
Notice that for any $z\in I$ and $a\in X$, we have $a*z=\lambda_a(z)-z$ in $X$, and in $\Br(X)$
\begin{equation*}
\begin{split}
a*_{\oplus}z=\mu_a(z)\oplus (-z)
= \lambda_a(z)-z+\frac{1}{2}[a, \lambda_a(z)]_++\frac{1}{2}[-z, \ \lambda_a(z)]_+
	\end{split}
\end{equation*}
Since $I$ is lambda-invariant and central, in the above equation all half-terms vanish. Thus,
\begin{equation}
\label{eqn: star op on gamma}
a*_{\oplus}z= \lambda_a(z)-z=a*z.
\end{equation}

Next assume that the left nilpotency class of $X$ is $r$. Then,  we have $X^{r+1}=0$ and by (\ref{eqn:higher product congruence}), we get $\Br(X)^{r+1}\subseteq I$. Since $I \subseteq X$, we have
\[
\underbrace{X*(X*\cdots *(X*(X}_{r \ \text{copies}}*I))\cdots )=0.
\] Then by equation (\ref{eqn: star op on gamma}), 
\[
\underbrace{\Br(X)*_{\oplus} \big( \Br(X)*_{\oplus}\cdots *_{\oplus}(\Br(X)*_{\oplus}(\Br(X)}_{r \ \text{copies}}*_{\oplus}I))\cdots \big)=0
\]
and hence $\Br(X)^{2r+1}=0$, which proves (1).

Now suppose that the left nilpotency class of $\Br(X)$ is $n$, that is $\Br(X)^{n+1}=0$. This implies that $X^{n+1} \subseteq I$, by (\ref{eqn:higher product congruence}). Since $I\subseteq \Br(X)$, by taking into account (\ref{eqn: star op on gamma}), we can conclude that
 \[
\underbrace{X*(X*\cdots *(X*(X}_{n \ \text{copies}}*I))\cdots )=0.
\] 
where $m$ is as above. Hence, $X^{2n+1}=0$, which yields (2).


\end{proof}

\begin{pro}
\label{pro:baer quotient}
Let $X$ be a skew brace of odd order such that the additive group $(X, +)$ is a nilpotent group of class at most $2$. Let $I$ be an ideal $X$. Then 
\[
\Br(X/I)=\Br(X)/I.
\]
\end{pro}

\begin{proof}
First of all, notice that since we have  $x\oplus I= x\cdot I=x+I$ for any $x\in X$, the underlying sets of $\Br(X/I)$ and $\Br(X)/I$ are the same.

Note that the additive group of $X/I$ has nilpotency class at most $2$. Thus we can consider the brace $\Br(X/I)=(X/I, \oplus, \cdot)$. Now for any $x+I, y+I \in X/I$, we have
\begin{equation*}
\begin{split}
	(x+I) \oplus (y+I)=&(x+I) + (y+I) + \frac{1}{2}[y+I, x+I]_+\\
	=& (x+y)+I+\frac{1}{2}([y,x]_++I)\\
	=& (x+y)+I+ \frac{1}{2}[y,x]_++\frac{1}{2}I\\
	=& (x+y+\frac{1}{2}[y,x]_+)+ I
\end{split}
	\end{equation*}
	
We next consider the quotient brace $\Br(X)/I$. Now pick $x\oplus I, y\oplus I \in \Br(X)/I$. Then by taking into account that $\gamma_2(X,+)$ is central, we have
\begin{equation*}
	\begin{split}
		(x\oplus I) \oplus (y\oplus I)=&(x\oplus y) \oplus I \\
		=& (x+y+\frac{1}{2}[y,x]_+) \oplus I\\
		=& (x+y+\frac{1}{2}[y,x]_+)+ I + \frac{1}{2}[I,\ x+y+\frac{1}{2}[y,x]_+]_+\\
		=&(x+y+\frac{1}{2}[y,x]_+)+ I
	\end{split}
\end{equation*}

Thus, $\Br(X/I)$ and $\Br(X)/I$ agree on addition.

Also for  $x\oplus I, y\oplus I \in \Br(X)/I$, we have
\[
(x\oplus I)\cdot (y\oplus I)=(x\cdot y) \oplus I= x\cdot y +I +\frac{1}{2}[I, x\cdot y]_+ =x\cdot y +I
\]
which is equal to $(x+I) \cdot (y+I)$ in $\Br(X/I)$.
\end{proof}

\begin{theo}
\label{thm:right nilpotency}
Let $X$ be a skew brace of odd order such that the additive group $(X, +)$ is a nilpotent group of class at most $2$. If $X$ is right nilpotent, then so is $\Br(X)$.
\end{theo}

\begin{proof}
Assume that $X$ is right nilpotent. We proceed by induction on $|X|.$
Let $n$ be the smallest integer such that $X^{(n+1)}=0$. Let $I= X^{(n)}$. Then $I\neq 0$ is an ideal of $X$ such that $I \subseteq \ker{\lambda}$. Notice that $X/I$ is also right nilpotent. Then by induction, $\Br(X/I)$ is right nilpotent. By Proposition \ref{pro:baer quotient}, we have $\Br(X/I)=\Br(X)/I$. Thus, $\Br(X)/I$ is right nilpotent, and so $\Br(X)^{(m)} \subseteq I$ for some $m$.

We will next prove that $(I*_{\oplus} \Br(X))*_{\oplus} \Br(X)=0$. By using the fact that $I \subseteq \ker{\lambda}$, we have
\[
\begin{aligned}
I*_{\oplus}\Br(X)&=\langle i *_{\oplus} x \mid i\in I, x\in X \rangle_{\oplus}\\
&= \langle  \lambda_i(x)-x+\frac{1}{2}[i, \lambda_i(x)]_++ \frac{1}{2}[-x, \lambda_i(x)]_+\rangle
&=\langle  \frac{1}{2}[i,x]_+ \mid i \in I, x\in X\rangle_+
\end{aligned}
\]
Then as $I$ is an ideal, we have  $I*_{\oplus}\Br(X) \in I \cap \gamma_2(X,+)$.
Now for any $z\in I\cap \gamma_2(X,+)$ and $y\in \Br(X)$, we have
\[
z*_{\oplus}y= \lambda_z(y)-y+ \frac{1}{2}[z, \lambda_z(y)]_++\frac{1}{2}[-y, \lambda_z(y)]_+=0
\]
since $z \in \ker{\lambda}$ and it is a central element.
It then follows that $(I *_{\oplus} \Br(X))*_{\oplus} \Br(X)=0$, and so $\Br(X)^{(m+2)}=0$.	
\end{proof}

The converse of Theorem \ref{thm:right nilpotency} is not true. 

\begin{exa}
	$X=\SmallSkewbrace(27,47)$ is not right nilpotent but $\Br(X)$ is right nilpotent.
\end{exa}

The following general result for two-sided skew braces will be used in Theorem \ref{theo:two-sided}.

\begin{lemma}
\label{lem:right-distributive-rho}
Let $X$ be a skew brace. Then $X$ is two-sided if and only
if, for every $c\in X$, the map
\[
\rho_c:(X,+)\longrightarrow (X,+),
\qquad
\rho_c(x)=x\cdot c-c
\]
is an automorphism of the additive group $(X,+)$.
\end{lemma}

\begin{proof}
Assume first that $X$ is two-sided. Then, for all $x,y,c\in X$, we have
\[
(x+y)\cdot c=x\cdot c-c+y\cdot c.
\]
Hence $$\rho_c(x+y)
=(x+y)\cdot c-c
=x\cdot c-c+y\cdot c-c=\rho_c(x)+\rho_c(y).$$
Thus $\rho_c$ is an endomorphism of $(X,+)$. Moreover, $\rho_c$ is bijective,
since it is the composition of the bijection $x\mapsto x\cdot c$
with the additive translation
$
z\mapsto z-c.
$
Therefore $\rho_c\in\Aut(X,+)$.

Conversely, assume that, for every $c\in X$, the map
$
\rho_c(x)=x\cdot c-c
$
is an automorphism of $(X,+)$. Then, for all $x,y,c\in X$, we have
$
\rho_c(x+y)=\rho_c(x)+\rho_c(y).$
That is,
\[
(x+y)\cdot c-c=(x\cdot c-c)+(y\cdot c-c).
\]
Adding $c$ on the right in the additive group gives
\[
(x+y)\cdot c=x\cdot c-c+y\cdot c.
\]
Hence $X$ is two-sided.
\end{proof}

\begin{theo}
\label{theo:two-sided}
Let $X$ be a two-sided skew brace of odd order such
that the additive group $(X,+)$ has nilpotency class at most $2$. Then
$
\Br(X)=(X,\oplus,\cdot)
$
is a two-sided brace.
\end{theo}

\begin{proof} Since $X$ is two-sided, we have
\[
(x+y)\cdot c=x\cdot c-c+y\cdot c
\]
for all $x,y,c\in X$. Equivalently by Lemma \ref{lem:right-distributive-rho}, for a fixed $c\in X$, the map
\[
\rho_c:(X,+)\longrightarrow (X,+),
\qquad
\rho_c(a)=a\cdot c-c,
\]  belongs to $\Aut(X,+)$.

In order to prove that $\Br(X)$ is two-sided, again by Lemma \ref{lem:right-distributive-rho}, we need to
prove that the map
\[
f_c:(X,\oplus)\longrightarrow (X,\oplus),
\qquad
f_c(a)=a\cdot c\ominus c,
\]
is an automorphism of $(X,\oplus)$. Indeed, we will prove that $f_c\in \Aut(X, +)$, and then the result follows from Lemma \ref{lem:baer}(4). Now write $a\cdot c=\rho_c(a)+c.$
Then, we have
\[
\begin{aligned}
f_c(a)
&=(a\cdot c)\oplus(-c)\\
&=(\rho_c(a)+c)\oplus(-c)\\
&=\rho_c(a)+\frac{1}{2}[-c,\ \rho_c(a)]_+.
\end{aligned}
\]
Now for all $a, b \in X$,
\[
\begin{aligned}
f_c(a+b)
&=\rho_c(a+b)+\frac{1}{2}[-c, \rho(a+b)]_+\\
&=\rho_c(a)+\rho_c(b)+\frac{1}{2}([-c, \rho(a)]_+ + [-c, \rho(b)]_+)\\
&=\rho_c(a)+\frac{1}{2}[-c, \rho(a)]_++\rho_c(b)+\frac{1}{2}[-c, \rho(b)]_+\\
&=f_c(a)+f_c(b).
\end{aligned}
\]
Thus, $f_c$ is an endomorphism of $(X,+)$. Since $X$ is finite, if we prove that $\ker f_c={0}$, then we conclude that $f_c \in \Aut(X,+)$.
Let $a\in \ker f_c$. Then this implies that
\[
\rho_c(a)=-\frac{1}{2}[-c, \rho_c(a)]_+ \in Z(X,+).
\]
Since $\rho_c(a)\in Z(X,+)$, we have $\frac{1}{2}[-c, \rho_c(a)]_+=0$, and hence $\rho_c(a)=0$. As $\rho_c \in \Aut(X,+)$, we conclude that $a=0$. Therefore, $\ker f_c={0}$, which completes the proof.
\end{proof}

\begin{theo}
Let $X$ be a bi-skew brace of odd order such that $(X,+)$ has nilpotency
class at most $2$. For $a\in X$, define
\[
g_a:X\longrightarrow X,\qquad
g_a(b)=b+\frac{1}{2}[b,a]_+.
\]
Then $\Br(X)$ is bi-skew if and only if $
g_a\in \Aut(X,\cdot)$
for every $a\in X$.
\end{theo}

\begin{proof}
Since $X$ is bi-skew, the reversed brace $(X,\cdot,+)$ is a skew brace.
Hence, for every $a\in X$, the map
\[
h_a:X\longrightarrow X,\qquad
h_a(b)=a^{-1}\cdot(a+b)
\]
is an automorphism of the group $(X,\cdot)$.

Now $\Br(X)$ is bi-skew if and only if the reversed structure
$
(X,\cdot,\oplus)
$
is a skew brace. Equivalently, for every $a\in X$, the map
\[
\theta_a:X\longrightarrow X,\qquad
\theta_a(b)=a^{-1}\cdot(a\oplus b)
\]
is an automorphism of $(X,\cdot)$.

Using the definition of the Baer addition, we have
\[
a\oplus b=a+b+\frac{1}{2}[b,a]_+=a+g_a(b).
\]
Therefore $\theta_a(b)=a^{-1}\cdot(a+g_a(b))=h_a(g_a(b))$
which gives that
\[
\theta_a=h_a\circ g_a.
\] 
Since $h_a\in \Aut(X,\cdot)$, the map $\theta_a$ is an automorphism of
$(X,\cdot)$ if and only if $g_a$ is an automorphism of $(X,\cdot)$.
Hence $\Br(X)$ is bi-skew if and only if $g_a\in\Aut(X,\cdot)$ for all
$a\in X$.
\end{proof}

\begin{theorem}
\label{thm:central-nilpotency-preserved}
Let $X$ be a skew brace of odd order such that
$(X,+)$ has nilpotency class at most $2$. If $X$ is centrally nilpotent, then
so is $\Br(X)$. 
\end{theorem}

\begin{proof}
Suppose that $X$ is centrally nilpotent. By \cite[Theorem 5.3]{tsa} $X$ is both left and right nilpotent. Then Theorem \ref{theo: left nilpotency} and Theorem \ref{thm:right nilpotency} yield that $\Br(X)$ is both left and right nilpotent. Since $(\Br(X),+)$ is abelian, again by \cite[Theorem 5.3]{tsa}, the result follows.
\end{proof}

\begin{theorem}\label{thm:solvability-preserved-derived}
Let $X$ be a skew brace of odd order such that
$(X,+)$ has nilpotency class at most $2$. If $X$ is solvable, then so is $\Br(X)$. 
\end{theorem}

\begin{proof}
We understand the solvability with respect to the commutator-derived series as in \cite{BEJP}. If $X=0$, there is nothing to prove. Assume now that $X\neq 0$ and put
\[
D=\partial(X)=[X,X]_X.
\]
Then $D$ is a solvable ideal of $X$ and
$
X/D
$
is an abelian skew brace. In particular, $X/D$ is solvable. By Proposition~\ref{prop:ideals-preserved}, $D$ is an ideal of $\Br(X)$.
Moreover, by Proposition \ref{pro:baer quotient}, we have
\[
\Br(X)/D=\Br(X/D)=
X/D
\]
since $X/D$ is already a brace. Hence $\Br(X)/D$ is solvable. Since solvable extensions of skew braces by solvable ideals are solvable, the brace $\Br(X)$ is solvable.
\end{proof}
\begin{theorem}\label{thm:supersolvability-preserved}
Let $X$ be a skew brace of odd order such that
$(X,+)$ has nilpotency class at most $2$. If $X$ is supersolvable, then
so is $\Br(X)$. 
\end{theorem}

\begin{proof}
Assume that $X$ is supersolvable. Then there exists a series of ideals
\[
0=I_0\leq I_1\leq \cdots \leq I_n=X
\]
such that, for every $1\leq j\leq n$, the factor
$
I_j/I_{j-1}
$
is an abelian brace of prime order. By Proposition~\ref{prop:ideals-preserved}, each $I_j$ is an ideal of $\Br(X)$.

We claim that the same series
\[
0=I_0\leq I_1\leq \cdots \leq I_n=\Br(X)
\]
is a supersolvable series for $\Br(X)$: Fix $1\leq j\leq n$. By Proposition~\ref{pro:baer quotient}, we have
\[
\Br(X)/I_{j-1}=\Br(X/I_{j-1}).
\]
Therefore the factor
$
I_j/I_{j-1}
$
appearing in $\Br(X)/I_{j-1}$ is the Baer transform of the corresponding
factor in $X/I_{j-1}$. Since $I_j/I_{j-1}$ is an abelian brace of prime order, its additive group is
cyclic of prime order and its additive and multiplicative operations coincide.
In particular, for all $u,v\in I_j/I_{j-1}$, we have
$
[v,u]_+=0.
$
Hence the Baer addition on this factor is
\[
u\oplus v
=
u+v+\frac{1}{2}[v,u]_+
=
u+v.
\]
Thus the Baer transform does not change this factor. The multiplication is
unchanged as well. Therefore

\[
0=I_0\leq I_1\leq \cdots \leq I_n=\Br(X)
\]
is an ideal series of $\Br(X)$ whose factors are abelian braces of prime
order. Hence $\Br(X)$ is supersolvable.
\end{proof}
\begin{remark} Consequently, although the Baer transform preserves the properties considered above in the forward direction, from $X$ to $\Br(X)$, the reverse implications for two-sidedness, bi-skewness, central nilpotency, solvability, and supersolvability require additional assumptions and should not be asserted in general. \end{remark}

\section{Applications}

In this section,  we first show that the Baer transform preserves internal direct product decompositions into ideals.

\begin{theo}\label{theo:direct-product-many}
Let $X$ and $\operatorname{Br}(X)$ be as in
Theorem~\ref{thm:baer-transform}. If
\[
X=I_1\oplus_X I_2\oplus_X \cdots \oplus_X I_t
\]
is an internal direct product decomposition of $X$ into nonzero ideals, then
\[
\operatorname{Br}(X)
=
I_1\oplus_{\operatorname{Br}(X)}I_2
\oplus_{\operatorname{Br}(X)}
\cdots
\oplus_{\operatorname{Br}(X)}I_t
\]
is an internal direct product decomposition of $\operatorname{Br}(X)$ into the
same nonzero ideals. 
\end{theo}

\begin{proof}
Since
\[
X=I_1\oplus_X I_2\oplus_X \cdots \oplus_X I_t
\]
is an internal direct product decomposition, we have
\[
I_1+I_2+\cdots+I_t=X
\]
and, for each $k$,
\[
I_k\cap \bigoplus_{j\neq k} I_j=0.
\]

We shall next prove that for any ideals $I, J \in X$, we have  $I+J=I\oplus_{} J$: Let $i\in I$ and $j\in J$. Then as both $I$ and $J$ are normal subgroups of $(X,+)$, we have $\frac{1}{2}[j,i]_+ \in I\cap J$. Then
\[
i\oplus j=i+j+\frac{1}{2}[j,i]_+ \in I+J
\]
and hence $I\oplus J \subseteq I+ J$. Conversely, observe that
\[
i+j= i\oplus \left( j- \frac{1}{2}[j,i]_+\right) \in I\oplus J
\]
and so $I+J\subseteq I\oplus J$.Thus, we have
\[
 I_1\oplus \cdots \oplus I_t =\Br(X)
\]
and
\[
I_k\cap \bigoplus_{\Br(X),\ j\neq k} I_j=0.
\]
Consequently
\[
\operatorname{Br}(X)
=
I_1\oplus_{\operatorname{Br}(X)}I_2
\oplus_{\operatorname{Br}(X)}
\cdots
\oplus_{\operatorname{Br}(X)}I_t
\]
is an internal direct product decomposition of $\operatorname{Br}(X)$ into the
same ideals.
\end{proof}

\begin{remark}
This gives a useful application of the Baer transform. The brace
$\operatorname{Br}(X)$ is of abelian type, and hence its ideal structure may be
more accessible than that of the original skew brace $X$. If one proves
that $\operatorname{Br}(X)$ is indecomposable, then
Theorem~\ref{theo:direct-product-many} immediately implies that $X$ is indecomposable. Thus the Baer transform provides a useful sufficient criterion for
indecomposability of skew braces.
\end{remark}

We finish with another application showing how the Baer transform can be used to
study automorphisms of finite $p$-skew braces. The idea is to pass from
the additive $p$-group of a skew brace to a suitable characteristic
subgroup of nilpotency class at most $2$, and then apply the Baer transform to
this subgroup. The resulting brace of abelian type detects the $p'$-subgroups
of the automorphism group of the original skew brace.

We shall use the following standard form of Thompson's critical subgroup
theorem.

\begin{lemma}[Thompson critical subgroup theorem]\label{lem:critical}
Let $P$ be a finite $p$-group. Then $P$ has a characteristic subgroup $C$,
called a critical subgroup, such that $C$ is nilpotent of class at most $2$.
Moreover, if $A$ is a $p'$-group of automorphisms of $P$ and $A$ acts
trivially on $C$, then $A$ acts trivially on $P$.
\end{lemma}

\begin{theorem}\label{thm:critical-subgroup-application}
Let $X$ be a skew brace such that $(X,+)$ is a
$p$-group, where $p$ is odd. Let $C$ be a Thompson critical subgroup of the
additive group $(X,+)$.  Then the kernel of the
restriction map
\[
\rho:\Aut(X)\longrightarrow \operatorname{Aut}(\operatorname{Br}(C)),
\qquad
\rho(\alpha)=\alpha|_C,
\]
is contained in $O_p(\Aut(X))$. In particular, if $O_p(\Aut(X))=1$, then $\rho$ is an injection.
\end{theorem}

\begin{proof}
Since $C$ is a Thompson critical subgroup of $(X,+)$, it
is characteristic in $(X,+)$ and hence $C$ is a strong left ideal of $X$. In particular, $(C,+,\cdot)$ is a subbrace of $X$. Note that as $p$ is odd and $(C,+)$ has nilpotency class at most $2$, we have the Baer transform of $(C,+,\cdot)$, that is,
$\Br(C)=(C,\oplus,\cdot)$.

Now let $\alpha\in \Aut(X)$. Since $C$ is characteristic in the additive
group $(X,+)$ and $\alpha$ is an automorphism of $(X,+)$, we have
$
\alpha(C)=C.
$
Hence the restriction $\alpha|_C$ is well-defined. Moreover, since $\alpha$ is a brace automorphism of $X$, it preserves both
operations $+$ and $\cdot$ and so
$
\alpha|_C\in \Aut(C,+,\cdot).
$

Note that $\alpha|_C\in (\Aut(C,\oplus))$ by Lemma \ref{lem:baer}, and so it follows that $\alpha|_C\in \Aut(\Br(C))$.

Consequently the restriction map
\[
\rho:\Aut(X)\longrightarrow \Aut(\Br(C)),\qquad
\rho(\alpha)=\alpha|_C
\]
is a well-defined group homomorphism.

Now suppose that $\rho(\alpha)=1$ for some $\alpha\in \Aut (X)$ of $p'$-order. It means that $\alpha$ induces the identity map on $\Br(C)$, and so  $\alpha$ induces the identity map on $C$. Then $\alpha|_C$ is the identity map, that is, it acts trivially on $C$. Since $\alpha$ has $p'$-order, $\alpha$ acts trivially on $(X,+)$ by Lemma~\ref{lem:critical},  and hence $\alpha=\mathrm{id}$. Thus, $\ker \rho$ is a $p$-group. Since $\ker \rho$ is a normal $p$-subgroup of $\Aut(X)$, we have $\ker \rho \leq O_p(\Aut(X))$. Then the result follows.

\end{proof}

\begin{corollary}\label{cor:aut-brD-pgroup}
Let $X$ be a skew brace such that $(X,+)$ is a
$p$-group, where $p$ is odd. Let $C$ be a Thompson critical subgroup of
$(X,+)$. If $\operatorname{Aut}(\operatorname{Br}(C))$ is a $p$-group, then
$\operatorname{Aut}(X)$ is a $p$-group as well.
\end{corollary}

\end{document}